\documentclass[a4paper, 10pt, twoside, notitlepage]{amsart}

\usepackage{amsmath,amscd}
\usepackage{amssymb}
\usepackage{amsthm}
\usepackage{comment}
\usepackage{graphicx}

\usepackage{mathrsfs}
\usepackage[ocgcolorlinks, linkcolor=blue]{hyperref}

\usepackage[ocgcolorlinks,linkcolor=blue]{hyperref}

\theoremstyle{plain}

\def\R {\mathbb{R}}

\def\p {\partial}

\newcommand{\mf}[1]{\mathbf{#1}}

\newtheorem{proposition}{Proposition}[section]
\newtheorem{theorem}[proposition]{Theorem}
\newtheorem*{theorem*}{Theorem}

\newtheorem{lemma}[proposition]{Lemma}

\theoremstyle{definition}
\newtheorem{definition}[proposition]{Definition}
\newtheorem{remark}[proposition]{Remark}
\numberwithin{equation}{section}

\title{Reconstruction of unknown cavity by single measurement}

\author[Lin, Nakamura and Wang]{Yi-Hsuan Lin$^{*}$, Gen Nakamura$^{\dagger}$ and Haibing Wang$^{\diamond}$}

\address{$^{*}$Department of Mathematics and Statistics, University of Jyv\"asky\"a, Finland.
	\newline\indent E-mail:{\tt yihsuanlin3@gmail.com}}
\address{$^{\dagger}$Department of Mathematics, Hokkaido University, Sapporo 060-0810, Japan.
	\newline\indent E-mail:{\tt \ nakamuragenn@gmail.com}}
\address{{$^{\diamond}$ School of Mathematics, Southeast University, Nanjing 210096, China.
		\newline
		\indent E-mail:{\tt\  hbwang@seu.edu.cn}}}

\begin{document}

	\maketitle

	\begin{abstract}
		
		In this paper we propose a domain sampling type reconstruction scheme for an inverse boundary value problem to identify an unknown cavity by single measurement on the accessible boundary of a known electric or heat conductive medium. Here the single measurement is to give single current or heat flux which can have a small support over the boundary, and we measure the corresponding voltage or temperature over the whole boundary.
		For this inverse boundary value problem, we adapted the single NRT (no response test) introduced by Luke and Potthast (\cite{Luke}) for inverse scattering problem and show that it can provide such a domain sampling type reconstruction scheme.

		\medskip
		
		\smallskip
		
		
	\end{abstract}


	\section{Introduction}\label{Section 1}
	
	
	We will first set up our inverse problem. To begin with let $\Omega \subset \R^n$ for $n=2,3$ be a bounded domain with $C^2$-smooth boundary $\p \Omega$. Physically $\Omega$ is a medium and it can be either homogeneous electric or heat conductive medium with conductivity $1$. Let $D\Subset \Omega$ be a cavity with $C^2$-smooth boundary $\p D$ such that $\Omega \setminus \overline{D}$ is connected. Then the voltage or temperature of electric or heat denoted by $u$ satisfies the following boundary value problem
	\begin{align}\label{Main equation}
	\begin{cases}
	\Delta u =0 & \text{ in }\Omega \setminus \overline{D}, \\
	u=f & \text{ on }\p \Omega, \\
	\p_\nu u=0 & \text{ on }\p D,
	\end{cases}
	\end{align}
	where $\nu $ is a unit normal vector on $\p D$ pointing into $\Omega\setminus \overline{D}$ and $f$ is taken from the $L^2$ based Sobolev space $H^{1/2}(\partial\Omega)$ of order $1/2$ on $\partial\Omega$ which is a specified voltage or temperature at $\partial\Omega$.
	
	It is well known that \eqref{Main equation} is well-posed. That is for any given $f\in H^{1/2}(\partial\Omega)$, there exists a unique solution $u=u_f$ in the $L^2$ based Sobolev space $ H^1(\Omega\setminus\overline D)$ of order $1$ in $\Omega\setminus\overline D$ to \eqref{Main equation} such that
	\begin{equation*}\label{well-posedness}
	\Vert u\Vert_{H^1(\Omega\setminus\overline D)}\le C \Vert f\Vert_{H^{1/2}(\partial\Omega)}
	\end{equation*}
	for some constant $C>0$ which does not depend on $f$ and $u$. Henceforth we call such a $C>0$ general constant, which may differ from place to place, but we will use the same notation $C$.

	Based on this well-posedness, one can calculate the Neumann derivative $\p_\nu u_f$ on $\partial\Omega$ which belongs to the dual space $H^{-1/2}(\partial\Omega)$ of $H^{1/2}(\partial\Omega)$, and this means that we can measure either electric current or heat flux on  $\partial\Omega$. The pair $\left\{f,\left. \partial_\nu u_f\right|_{\partial\Omega}\right\}$ with the unit normal $\nu$ of $\partial\Omega$ directed outside $\Omega$ is called a Cauchy data. Throughout this paper, we assume that the boundary data $f$ on $\p \Omega$ is a \emph{non-constant} function.
	Then our inverse boundary value problem can be stated as follows.
	
	\medskip\noindent
	{\bf Inverse Problem}\newline
	Given a set of Cauchy data $\left\{f,\left. \partial_\nu u_f\right|_{\partial\Omega}\right\}$ taken as our measurement, identify $D$ from this measurement.

	\begin{remark}\label{remark for IP}${}$
		\begin{itemize}
			\item[1.] As described in the abstract, it is physically more natural to give either electric current or heat flux at $\partial\Omega$ as an input. That is to replace $u=f$ at $\partial\Omega$ in \eqref{Main equation} by $\partial_\nu u=f\in H^{-1/2}(\partial\Omega)$ at $\partial\Omega$. In that case we do not have the uniqueness of the corresponding boundary value problem but the solution is unique up to constant. However since we are taking a set of Cauchy data and this specifies the Dirichlet data on  $\partial\Omega$, we have the same situation as in \eqref{Main equation}.
			
			\item[2.] This inverse problem is physically meaningful for the spatial dimensions $n=1,2,3$. But we excluded the case $n=1$, because we want to have geometrically uniform descriptions and this case can be handled much easier. The problem can be considered for more general elliptic equations of divergence form and isotropic static elasticity equation.
			
			\item[3.] In stead of having Neumann boundary condition for $\partial D$, we could have Dirichlet boundary condition for $\partial D$. In this case, $\partial D$ physically means that it is an earthing boundary for an electric conductive medium and a cooling boundary with $0$ relative temperature for a heat conductive medium. In relation with the above item 1, the boundary value problem corresponding to \eqref{Main equation} is uniquely solvable whether we give an inhomogeneous Dirichlet boundary condition or an inhomogeneous Neumann boundary condition with data $f$.
		\end{itemize}
	\end{remark}
	
	The uniqueness of this inverse problem has been already known very early for example from the proof given for the uniqueness of identifying an unknown rigid inclusion inside an isotropic elasticity medium \cite{Ang2}. Also the stability estimate for the identification is known for the conductivity equation \cite{Alessandrini} and even for the isotropic elasticity system \cite{Higashimori,Morassi}. Hence we are particularly interested in giving a reconstruction.
	
	\medskip
	Our main result is the following.
	\begin{theorem}\label{main result} There is a domain sampling type reconstruction method for the aforementioned inverse problem.
		Its details will be given in Section \ref{single NRT}.
	\end{theorem}
	
	Our reconstruction method is the single wave no response test adapted to the inverse boundary value problem. The no response test was introduced by Luke-Potthast in \cite{Luke} for the inverse acoustic scattering problem to identify a scatterer such as a sound soft or sound hard obstacle. There are single wave no response test and multiple waves no response test. The corresponding measurements are the far field of the scattered wave generated by one incident plane wave and the scattering amplitude generated by multiple incident waves, respectively. Here it should be remarked that the multiple incident waves mean infinitely many incident waves. The multiple waves no response test can recover the scatterer, but the single wave no response test in general can only recover the scattering support which gives a lower estimate of the scatterer. For further information about the no response test for inverse scattering problems see \cite{Nakamura} and \cite{Potthast}. We will refer the single wave no response test adapted to the inverse boundary value problem by NRT.
	
	It is possible to obtain a result similar to Theorem \ref{main result} for more general equations such as the conductivity equation with anisotropic and heterogeneous conductivity, and also for the static elasticity equation with isotropic and heterogeneous elasticity tensor.
	The same is true for an unknown $D$	with Dirichlet boundary condition at $\partial D$ (see Remark \ref{remark for IP}).
	
	Since there is a huge literature on the reconstruction methods for our inverse problem, we only give some major reconstruction methods by citing one paper which we came across with strong interest. So we ask the readers to consult the literature there in and make further search
	to collect more information about the methods. They are iterative method using the domain derivative \cite{Kress}, topological derivative method \cite{Bonnet}, level set method \cite{Burger} and quasi-reversibility method \cite{Bourgeois}. Let us locate our reconstruction method which we called NRT among the aforementioned reconstruction methods. We can say that our NRT is a  quite simple mathematically rigorous method compared with the other methods. However we haven't studied the convergence of our method for noisy data and its numerical performance.
	We expect that our method will play some role to find a good initial guess for the iterative method such as the regularized least square method and regularized quasi-Newton type method.	
	
	The rest of this paper is organized as follows.
	In Section \ref{single NRT} we first provides some preliminary observation which is useful to introduce NRT. Then we restate our main theorem more precisely in terms of NRT. Section \ref{convergence} is devoted to proving the convergence of NRT.

	\section{NRT and its preliminary observation}\label{single NRT}
	
	In this section we will give Theorem \ref{main result} more precisely in terms of NRT.
	We first give a preliminary observation which can smoothly lead us to introduce NRT adapted to our inverse problem. To begin with let $u$ be the solution to \eqref{Main equation} and $v \in H^1(\Omega)$ be the solution to the boundary value problem
	\begin{align*}
	\Delta v =0\,\,\text{ in }\,\,\Omega,\,\,v=f\,\, \text{ on }\,\,\p\Omega.
	\end{align*}
	Then $w:=u-v\in H^1(\Omega\setminus\overline D)$
	satisfies
	\begin{align}\label{equation w}
	\begin{cases}
	\Delta w =0 & \text{ in }\,\,\Omega\setminus \overline{D},\\
	\p_\nu w =-\p _\nu v &\text{ on }\,\,\p D, \\
	w=0 & \text{ on }\,\,\p \Omega.
	\end{cases}
	\end{align}
	
	Now for $g\in H^{1/2}(\partial\Omega)$, let $z=z_g$ be the solution to the boundary value problem
	\begin{align}\label{equation z_g}
	\Delta z_g =0\,\,\text{ in }\,\,\Omega,\,\,z_g =g\,\,\text{ on }\,\,\p\Omega,
	\end{align}
	then we can prove the following identity.
	
	\begin{lemma}\label{Lemma key identity}
		\begin{align}\label{key identity}
		\int_{\p \Omega }\p_\nu w \cdot g dS =- \int _{\p D} u\cdot \p_\nu z_g dS,
		\end{align}
		where $\nu $ denotes the outer unit normal of $\partial\Omega$ and $\partial D$ directed outside
		$\Omega$ and $D$, respectively.
	\end{lemma}
	
	
	\begin{proof}
		By using equations and boundary conditions of $w$ and $z_g$, we have
		\begin{align}\label{cal 1}
		0 = &\notag \int_{\Omega\setminus \overline{D}} \Delta w \cdot z_g dx=\int_{\p (\Omega \setminus \overline{D})}\p_\nu w \cdot z_g dS -\int_{\Omega \setminus\overline{D}}\nabla w \cdot \nabla z_g dx\\
		= & \notag \int_{\p \Omega }\p _\nu w \cdot z_g dS-\int _{\p D} \p_\nu w \cdot z_g dS -\int_{\p(\Omega \setminus\overline{D})}w \cdot \p_\nu z_g dS\\
		= & \int_{\p \Omega }\p _\nu w \cdot z_g dS+\int _{\p D} \p_\nu v \cdot z_g dS+\int_{\p D} (u-v) \cdot \p_\nu z_g dS
		\end{align}
		where $\nu $ for the integrals on $\partial(\Omega\setminus\overline D)$  is the outer unit normal of $\p (\Omega\setminus\overline D)$. Note that by the Green formula we have
		\begin{align}\label{cal 2}
		\int_{\p D}v \cdot \p_\nu z_g dS -\int _{\p D}\p _\nu v\cdot z_g dS=-\int _D \left(v\Delta z_g-z_g \Delta v\right)dx =0.
		\end{align}
		Then combining \eqref{cal 1} with \eqref{cal 2}, we immediately have \eqref{key identity}.

	\end{proof}
	
	Next we introduce test domains and an indicator function as follows.
	
	\begin{definition} We call any subdomain $G\Subset\Omega$ a test domain if it satisfies the condition that $\Omega\setminus\overline G$ is connected. Then, for any test domain $G$ and $\epsilon>0$, we define $I_\epsilon(G)$ by
		\begin{align*}
		I_\epsilon (G) := \displaystyle\sup_g\left\{\left|\int_{\p \Omega}\p_\nu w \cdot g dS\right|:g\in H^{1/2}(\partial\Omega)\,\,\text{such that}\,\,\|z_g\|_{H^1(G)}<\epsilon \right\},
		\end{align*}
		where $w$ and $z_g$ are the solutions of \eqref{equation w} and \eqref{equation z_g}, respectively. Clearly $I_\epsilon(G)$ is non-negative and monotone decreasing as $\epsilon\searrow0$. Based on this we define the domain sampling indicator function $I(G)$ for a test domain $G$ by
		\begin{align*}
		I(G):=\lim _{\epsilon\searrow0}I_\epsilon (G).
		\end{align*}
	\end{definition}
	
	By using this indicator function, we classify test domains as follows.
	
	\begin{definition}
		A test domain $G$ is called \emph{positive} or \emph{no-response} if $I(G)=0$.
		Further we denote the set of all positive test domain by $\mathcal{P}$. That is
		$\mathcal{P}:=\{G: \ G \text{ is a positive test domain}\}.$
	\end{definition}
	
	\begin{remark}${}$
		The meaning of no-response is as follows. If we mask the cavity $D$ by a test domain $G$ i.e. $\overline D\subset\overline G$, then we cannot have any response i.e. $I(G)=0$ (see the proof of next Theorem \ref{Thm Main} in the next section).

	\end{remark}
	
	Now we are able to state our main result in terms of NRT as follows.
	
	\begin{theorem}[reconstruction formula]\label{Thm Main}
		The cavity $D$ can be reconstructed as
		\begin{align}\label{Reconstruction formula}
		\cap_{G\in\mathcal{P}}\ \overline G =\overline{D}.
		\end{align}
	\end{theorem}
	
	The proof of this theorem will be given in the next section. Before closing this section, we give a modified version of Theorem \ref{Thm Main}.
	
	\begin{remark}\label{modified reconstuction formula}
		Since for any $\epsilon>0$,
		\begin{align*}
		&\notag \quad g\in H^{1/2}(\partial\Omega),\,\,\Vert z_g\Vert_{H^1(G)}<\epsilon
		\\
		\Longleftrightarrow & \quad \epsilon^{-1}g\in H^{1/2}(\p \Omega),\,\,\Vert z_{\epsilon^{-1}g}\Vert_{H^1(G)}<1,
		\end{align*}
		we can just take $I_1(G)$ as an indicator function for a test domain $G$. That is
		\begin{equation*}
		\displaystyle\cap_{G\in\mathcal{Q}}\,\overline{G}=\overline D,
		\end{equation*}
		where $\mathcal{Q}=\{G\,:\,\text{test domain}\,\,I_1(G)<\infty\}.$
	\end{remark}

	\section{Convergence proof of NRT}\label{convergence}
	We will give a proof of Theorem \ref{Thm Main} without passing through Remark \ref{modified reconstuction formula}. In order to prove \eqref{Reconstruction formula}, it suffices to show that
	\begin{itemize}
		\item[1.] $\overline{D}\subset\overline G$ implies that $I(G)=0$.
		\item[2.] $\overline{D}\not \subset G$ implies that $I(G)=\infty$.
	\end{itemize}
	In fact by the definition of $\mathcal{P}$, if these statements hold, then we have
	\begin{itemize}
		\item[1'.] $\overline D\subset\overline G$ implies $G\in\mathcal P$.
		\item[2'.] $\overline D\not\subset\overline G$ implies $G\not\in\mathcal{P}$.
	\end{itemize}
	
	First, we show the statement 1. Suppose we have $\overline{D}\subset\overline G$. For any $\epsilon>0$, let $g\in H^1(\Omega)$ satisfy $\|z_g\|_{H^1(G)}<\epsilon$. By \eqref{key identity} and $\Vert\partial_\nu z_g\Vert_{H^{-1/2}(\partial D)}\le C\Vert z_g\Vert_{H^1(G)}$ with a general constant $C>0$, we have
	$$
	\left|	\int_{\p \Omega }\p_\nu w \cdot g dS \right|= \left|\int _{\p D} u\cdot \p_\nu z_g dS\right|\leq C\|u\|_{H^{1/2}(\p D)}\|z_g \|_{H^1(G)}\leq C\epsilon
	$$
	for a general constant $C>0$. Hence $I_\epsilon (G)
	\leq C\epsilon$ so that $I(G)=\lim_{\epsilon\searrow 0}I_\epsilon (G) =0$.
	
	Next we show the statement 2. Since the proof for the case $n=2$ is almost the same, we confine to the case $n=3$. Suppose we have $\overline{D}\not \subset\overline G$.
	Then there exists a point $y_0\in\partial D\setminus\overline G$ and a small and narrow cylinder like open neighborhood $N_{y_0}$ of $y_0$ sitting on $\partial D$ with symmetric axis $\nu_{y_0}$ and a flat top surface such that $\overline{N_{y_0}}\cap\overline G=\emptyset$, where $\nu_{y_0}$ is the outer unit normal $\nu_{y_0}$ directed outside $D$. Since $\partial_\nu u=0$
	on $\partial D$, $u=f\in H^{1/2}(\partial\Omega)$ on $\partial \Omega$ non-constant $f\in H^{1/2}(\partial\Omega)$ and $\Omega\setminus\overline D$ is connected, $u$ cannot vanish in any open subset of $\partial D$ due to the unique continuation property for solutions of Laplace equation and the regularity up to $\partial D$ of $u$ giving $u\in C^0(\partial D)$. Hence we can assume that $|u(x)|\ge\delta,\,y\in N_{y_0}\cap\partial D$ for some constant $\delta>0$.
	
	To proceed further we will consider a singular solution of the Laplace equation. For example let $E(x,y)$ be the fundamental solution $E(x,y):=(4\pi|x-y|)^{-1},\,x\not=y$ of $-\Delta$ in $\R^3$ and for any fixed $\mf{a}\in \R^3$, we take our singular solution of the Laplace equation as
	\begin{align*}
	F_{\mathbf{a}}(x,y):=(\mathbf{a}\cdot\nabla_x)E(x,y)=-\dfrac{(x-y)\cdot \mathbf{a}}{4\pi|x-y|^3},\ x\not=y \text{ with }\,\mf{a}=\nu_{y_0}.
	\end{align*}
	and estimate $\liminf_{y\to y_0}\left|\int_{\partial D}u(x)\cdot\p_{\nu_x}F_{\mf{a}}(x,y)dx\right|$ from below, where $y$ moves along the axis of the cylinder $N_{y_0}$. It is easy to see that
	\begin{equation}\label{bounded}
	\left|\int _{\p D \setminus N_{y_0}}u(x)\cdot \p _{\nu_x}F_{\mf{a}}(x,y)dx\right|\,\,\text{is bounded}
	\end{equation}
	as $y\rightarrow y_0$ along the axis of the cylinder $N_{y_0}$. By a direct computation we have
	\begin{equation}\label{integrand}
	4\pi\p_{\nu_x}F_\mf{a}(x,y)=-\dfrac{\mf{a}\cdot \nu_x}{|x-y|^3}+3\dfrac{[(x-y)\cdot\nu_x][(x-y)\cdot \mf{a}]}{|x-y|^5}.
	\end{equation}
	Write
	\begin{equation}\label{vector identity}
	\begin{array}{ll}
	(x-y)\cdot \nu_x = (x-y_0)\cdot \nu_{y_0}-(y-y_0)\cdot \nu _{y_0} +(\nu_x-\nu_{y_0})\cdot (x-y)\\
	(x-y)\cdot\mf{a}=(x-y_0)\cdot\nu_{y_0}-(y-y_0)\cdot\nu_{y_0},
	\end{array}
	\end{equation}
	where $x,y_0 \in \p D$ and $\nu _{y_0}$ is a unit normal vector pointing to $\Omega\setminus \overline{D}$. Further recall the following estimates:
	\begin{align*}
	|\nu_{y_0} \cdot (x-y_0)|\leq L|x-y_0|^2 \quad \text{ and }\quad |\nu_x - \nu_{y_0}|\leq L|x-y_0|,
	\end{align*}
	with some constant $L>0$ for any $x,y_0\in \p D$ (see \cite[Theorem 2.2]{colton2013integral}). Hence the second term of the right hand side of
	\eqref{integrand} is small compared with the first term of the right hand side of \eqref{integrand} for $x\in N_{y_0}\cap\partial D$ and $y\rightarrow y_0$ along the axis of the cone $N_{y_0}$. Then taking into account that the first term of the right hand side of \eqref{integrand} is negative for $x\in N_{y_0}\cap\partial D$ and $y\rightarrow y_0$ along the axis of the cylinder $N_{y_0}$, we have
	\begin{align}\label{liminf 1}
	\liminf_{y\to y_0}\left|\int_{N_{y_0}}u(x)\cdot\p_{\nu_x}F_{\mf{a}}(x,y)dx\right|\geq \int_{N_{y_0}}|u(x)|\liminf_{y\to y_0}|\p_{\nu_x}\cdot F_{\mf{a}}(x,y)|dx=\infty
	\end{align}
	by Fatou's lemma.
	Therefore combining \eqref{bounded} and \eqref{liminf 1}, we have obtained the following observation:
	\begin{equation}\label{liminf}
	\liminf_{y\to y_0}\left|\int_{\partial D}u(x)\cdot\p_{\nu_x}F_{\mf{a}}(x,y)dx\right|=\infty\,,
	\end{equation}
	where $y$ moves along the axis of the cylinder $N_{y_0}$.
	
	Now based on this observation, for any fixed $\epsilon>0$, we want to find a sequence $g_j\in H^1(\partial\Omega),\,j=1,2,\cdots$ such that
	\begin{equation*}
	\Vert z_{g_j}\Vert_{H^1(G)}<\epsilon,\,\,
	\displaystyle\lim_{j\rightarrow\infty}\left|
	\int_{\partial\Omega}\partial_\nu w\cdot g_j\,dx\right|=\displaystyle\liminf_{j\rightarrow\infty}\left|\int_{\partial D} u\cdot\partial_\nu z_{g_j}\,dx\right|=\infty,
	\end{equation*}
	which immediately implies $I(G)=\infty$.
	
	Let $y_j,\,j=1,2,\cdots$ be a sequence along the axis of $N_{y_0}$ such that $y_j\rightarrow y_0$ as $j\rightarrow\infty$. Further let $A_j$ be a domain such that $D\cup G\Subset A_j\Subset \Omega$ and $y_j\notin A_j$.
	Then consider the sequence of functions
	$$
	\mathcal{G}_j(x):= \dfrac{\epsilon}{2}(\|F_{\mf{a}}(\cdot,y_j)\|_{H^1(G)}+1)^{-1}F_{\mf{a}}(x,y_j),\,\,j=1,2,\cdots.
	$$
	Note that each $\mathcal{G}_j\in H^1(A_j)$ and it satisfies $\Delta\mathcal{G}_j=0$ in $A_j$. By the Runge approximation property for the solutions of Laplace equation (see \cite{lax}), we can approximate each $\mathcal{G}_j$ in $A_j$ by a function $U_j\in H^1(\Omega)$ such that $\Delta U_j =0$ in $\Omega$. Let $\epsilon/4$ be the discrepancy of approximation so that we have
	$\left\|U_j -\mathcal{G}_j \right\|_{H^1(\overline{D\cup G})}<\dfrac{\epsilon}{4}$.
	
	Now take $g_j=U_j\big|_{\p \Omega}$ and denote $z_{g_j}:= U_j$ in $\Omega$. Then $g_j\in H^1(\Omega)$ and it satisfies $\Delta z_{g_j}=0$ in $\Omega$, $z_{g_j}=g_j$ on $\partial\Omega$. Observe that we have
	$$
	\|z_{g_j} \|_{H^1(G)}\leq \|z_{g_j}-\mathcal{G}_j\|_{H^1(G)}+\|\mathcal{G}_j\|_{H^1(G)}<\dfrac{\epsilon}{4}+\dfrac{\epsilon}{2}<\epsilon,
	$$
	and
	\begin{align}\label{liminf 2}
	\left|\int_{\partial D}u\cdot \p_{\nu} z_{g_j}dS\right|\geq &\notag \left|\int_{\partial D}u\cdot \p_{\nu} \mathcal{G}_j dS\right|-\int_{\partial D}|u| \left|\p_{\nu}\left(z_{g_j}-\mathcal{G}_j\right) \right|dS \\
	\geq & \dfrac{\epsilon}{2}(\left\|F_{\mf{a}}(\cdot,y_j)\right\|_{H^1(G)}+1)^{-1}\left|\int_{\partial D}u\cdot \p_{\nu}F_{\mf{a}}(\cdot,y_j)dS\right|-C\epsilon,
	\end{align}
	for some general constant $C>0$. Here note that $\Vert F_{\mf{a}}(\cdot,y_j)\Vert_{H^1(G)}$ is bounded as $j\rightarrow\infty$. Hence by taking $\displaystyle\liminf_{j\rightarrow\infty}$ on the both sides of \eqref{liminf 2}, we immediately have
	$$
	\liminf_{j\to\infty}\left|\int_{\partial D}u\cdot \p_{\nu} z_{g_j}dS\right|=\infty
	$$
	from \eqref{liminf}.
	This completes the proof.
	
	\begin{remark}${}$
		\newline
		In the proof of Theorem \ref{main result} we have only used the following ingredients.
		\begin{itemize}
			\item [1.]
			The Green formula to derive the key identity \eqref{key identity}.
			\item[2.] The unique continuation property to say that the solution giving the measurement doesn't vanish in an open subset of the boundary of an unknown cavity.
			\item[3.] The Runge approximation of some singular solution for showing the blow up of the indicator function. Here the Runge approximation is coming from the unique continuation property.
		\end{itemize}
		Hence a result similar to the theorem can be obtained for more general equations and systems as far as we can have these ingredients.
		
	\end{remark}		
	
	\section*{Acknowledgment}
	
	The first author acknowledges the supports from the Finnish Centre of Excellence in Inverse Modelling and Imaging (Academy of Finland grant 284715) and the Academy of Finland project number 309963. The second author acknowledges the supports from Grant-in-Aid for Scientific Research 15K21766 and 15H05740 of JSPS. The third author is supported by National Natural Science Foundation of China (No. 11671082) and Qing Lan Project of Jiangsu Province.

	\bibliographystyle{abbrv}
	\bibliography{ref}

\end{document}